\newcommand{\ms}{\medskip}

\documentclass[12pt]{amsart}
\usepackage{amssymb} 
\usepackage[mathscr]{eucal} 
\usepackage{xypic}   
\usepackage{amsmath} 
\usepackage{epsfig}
\usepackage{amscd}
\usepackage{mathrsfs}

\def\Q{{\mathbb Q}}
\def\Z{{\mathbb Z}}
\def\R{{\mathbb R}}
\def\C{{\mathbb C}}
\def\P{{\mathbb P}}
\def\Gal{\mathrm{Gal}}

\def\Hom{\mathrm{Hom}}
\def\Mat{\mathrm{Mat}}
\def\Sin{\mathrm{Sin}}

\def\Spec{\mathrm{Spec}}
\def\pro{\mathrm{pro}}
\def\Map{\mathrm{Map}}
\def\I{{I}}

\def\G{{\mathbb G}}
\def\s{{\mathscr S}}
\def\t{{\mathcal T}}
\def\HH{{\mathfrak H}}
\def\BB{{\mathrm B}}
\def\Nn{{\mathcal N}}

\def\Isom{\mathrm{Isom}}

\def\sk{\mathrm{sk}}
\def\cosk{\mathrm{cosk}}

\def\SL{\mathrm{SL}}
\def\SP{\mathrm{Sp}}
\def\CSP{\mathrm{CSp}}

\def\IM{\mathrm{Im}}

\def\Pic{\mathrm{Pic}}
\def\RR{\mathrm{R}}
\def\O{{\mathcal O}}
\def\P{{P}}

\def\hf{{\mathfrak S}}
\def\vf{{\mathfrak h}}

\def\P{{\mathfrak P}}
\def\i{{\mathfrak i}}
\def\c{{\mathfrak c}}

\def\T{{\mathfrak T}}
\def\E{{\mathfrak E}}
\def\m{{\mathscr M}}
\def\a{{\mathscr A}}
\def\h{{\mathcal H}}
\def\O{{\mathcal O}}
\def\I{{\mathcal I}}
\def\r{{\mathcal R}}
\def\l{{\mathcal L}}
\def\e{{\}_{et}}}

\theoremstyle{definition}
\newtheorem{theorem}{Theorem}[section]
\newtheorem*{itheo}{Theorem}

\newtheorem{corollary}[theorem]{Corollary}
\newtheorem{definition}[theorem]{Definition}

\theoremstyle{remark}

\numberwithin{equation}{section}

\begin{document}

\title[\'Etale Homotopy of Moduli of abelian schemes]
{\'Etale Homotopy Types of Moduli Stacks of polarised abelian schemes}

\author[Paola Frediani]
{Paola Frediani}
\address{Dipartimento di Matematica\\
Universit\`a degli Studi di Pavia\\
Via Ferrata 1, I-27100 Pavia, Italy}
\email{frediani@dimat.unipv.it}

\author[Frank Neumann]
{Frank Neumann}
\address{Department of Mathematics\\
University of Leicester\\
University Road, Leicester LE1 7RH, England, UK}
\email{fn8@le.ac.uk}



\keywords{\'etale homotopy theory, algebraic stacks, moduli of
abelian schemes, principally polarised abelian varieties, algebraic curves\\
{\it Dedicated to Professor Ronald Brown on the occasion of his 80th birthday and to the memory of Alexander Grothendieck.}}

\begin{abstract}
{We determine the Artin-Mazur \'etale homotopy types of moduli stacks of polarised
abelian schemes using transcendental methods and derive some arithmetic properties 
of the \'etale fundamental groups of these moduli stacks. Finally we analyse the Torelli morphism
between the moduli stacks of algebraic curves and principally
polarised abelian schemes from an \'etale homotopy point of view.}
\end{abstract}

\maketitle



\section*{Introduction}

The use of the Artin-Mazur machinery of \'etale homotopy theory \cite{AM} for Deligne-Mumford stacks was pioneered by Oda \cite{O} following ideas of Grothendieck \cite{G} in order to study arithmetic homotopy types of moduli stacks of algebraic curves and to relate them to geometric representations of the absolute Galois group $\Gal(\bar{\Q}/ \Q)$.
Let $\m_{g,n}$  be the moduli stack of smooth algebraic curves of genus $g$ with $n$ distinct ordered points and $2g+ n > 2$. 
Oda determined the Artin-Mazur  \'etale homotopy type of the moduli stack $\m_{g, n} \otimes \bar{\Q}$ over the big \'etale site of schemes over $\bar{\Q}$. Using transcendental methods by comparing it with the complex analytic situation Oda showed that the \'etale homotopy type of $\m_{g, n} \otimes \bar{\Q}$ is given as the profinite Artin-Mazur completion of the Eilenberg-MacLane space $K(\Map_{g, n}, 1)$, where $\Map_{g, n}$ is the Teichm\"uller modular or mapping class group of compact Riemann surfaces of genus $g$ with $n$ punctures. Oda's results allow to analyse geometric Galois actions as they give rise to a short exact sequence relating the absolute Galois group $\Gal(\bar{\Q}/\Q)$, the \'etale fundamental group  $\pi^{et}_{1}(\m_{g, n} \otimes \Q, x)$ of the moduli stack of algebraic curves over the rationals $\Q$ and the profinite completion $\Map_{g,n}^{\wedge}$ of the mapping class group $\Map_{g,n}$ (cf. \cite{Ma}, \cite{SL}, \cite{Mk}).

In \cite{fn} the authors followed up on the theme of Oda and determined the \'etale homotopy types of related moduli stacks of algebraic curves with prescribed symmetries including moduli stacks of hyperelliptic curves. Later Ebert and Giansiracusa in \cite{EG} determined also the \'etale homotopy types of the Deligne-Mumford compactifications  $\bar{\m}_{g,n}$ being again Deligne-Mumford stacks which classify stable algebraic curves of genus $g$ with $n$ distinct ordered points and which are compactifications of the moduli stacks $\m_{g,n}$.

In this article we will now analyse \'etale homotopy types of moduli stacks $\a_D$ of abelian schemes with polarisations of a general given type $D$ and related moduli stacks $\a_{D, [N]}$ of polarised abelian schemes with level structures. Special cases include the moduli stack $\a_g$ of principally polarised abelian varieties and the moduli stack $\m_{ell}$ of elliptic curves. Moduli stacks of abelian schemes and their compactifications play a fundamental role in algebraic geometry and number theory. 

It turns out that a similar theorem like that of Oda for the moduli stacks of algebraic curves holds with some modifications also for the moduli stacks $\a_{D}$ and $\a_{D, [N]}$ of polarised abelian schemes with and without level structures. 

More precisely, let $g\geq 1$ be a positive integer and $D=(d_1, d_2, \ldots d_g)$ be a tuple of integers such that $d_1|d_2|\ldots |d_g$.
The moduli stack $\a_D$ of abelian schemes with polarisations of type $D$ is known to be an algebraic Deligne-Mumford stack.  It has also a complex analytic uniformisation $\a_D^{an}$ given as a complex analytic Deligne-Mumford quotient stack or complex analytic orbifold of the form $\a^{an}_D=[\HH_g/\SP_D(\Z)]$. Here $\HH_g$ is the Siegel upper half space of genus $g$, which carries a natural action of the discrete group  $\SP_D(\Z) = \{\gamma\in \Mat(2g, \Z) | {}^t\gamma J_D\gamma = J_D\} $, where the matrix $J_D$ is given as
$$J_D =\left(
\begin{array}{cc}
0 & D \\
  &  \\
-D & 0\\
\end{array}
\right),
$$ 
and $D$ is the diagonal matrix with entries  $(d_1,...,d_g)$

In order to determine the \'etale homotopy type $\{\a_D\otimes \bar{\Q}\e$ we compare the algebraic with the associated complex analytic situation.  As the Siegel upper half space $\HH_g$  is a contractible topological space this allows to determine
the classical homotopy type of the stack $\a_D^{an}$ and using a general comparison theorem comparing \'etale
and classical homotopy types essentially due to Artin-Mazur \cite{AM} we can derive our main result

\begin{itheo}
\label{etho}
Let $g\geq 1$ be a positive integer and $D=(d_1, d_2, \ldots d_g)$ be a tuple of integers with $d_1|d_2|\ldots |d_g$. 
There is a weak homotopy equivalence of pro-simplicial sets
$$\{\a_D \otimes\bar{\Q} \}^{\wedge}_{et} \simeq K(\SP_D (\Z), 1)^{\wedge}.$$
where $^{\wedge}$ denotes Artin-Mazur profinite completion.
In particular for principal polarisations we have
$$\{\a_g \otimes\bar{\Q} \}^{\wedge}_{et} \simeq K(\SP (2g, \Z), 1)^{\wedge}.$$
\end{itheo}

An important special case of $\a_g$ is given by the moduli stack $\m_{ell}$ of elliptic curves. Here $g=1$ and we have $\{\m_{ell} \otimes\bar{\Q} \}^{\wedge}_{et} \simeq K(\SL (2, \Z), 1)^{\wedge}.$ The arithmetic properties of the Deligne-Mumford stack $\m_{ell}$ were first studied by Deligne and Rapoport \cite{DR} and it plays a fundamental role also in the systematic construction of elliptic cohomology theories \cite{Lu}. 

We also derive similar results for the moduli stacks $\a_{D, [N]}$ which are modifications of our main result using appropriate congruence subgroups $\Gamma_D(N)$ of $\SP_D (\Z)$ which reflect the particular level $N$ structures. 

From these general results we also obtain short exact sequences involving the \'etale fundamental groups of the moduli stacks $\a_D$ and $\a_{D, [N]}$ relating them to the profinite completions of the discrete groups $\SP_D (\Z)$ and $\Gamma_D(N)$ and the absolute Galois group
$\Gal(\bar{\Q}/\Q)$. In this way we extend the results of Oda to other interesting moduli stacks which allow for further interpretations of geometric Galois actions. As a bonus we can directly relate the \'etale homotopy types of the moduli stacks $\m_g$ of algebraic curves with the \'etale homotopy types of the moduli stacks $\a_g$ of principally polarised abelian schemes via the stacky Torelli morphism $j: \m_g\rightarrow \a_g$ between these moduli stacks, which is induced by taking Jacobians of algebraic curves. We also make similar observations again for the moduli stacks $\m_{g, [N]}$ and $\a_{g, [N]}$ which take into account the level structures.

\ms\noindent The paper is organized as follows: In the first section we recall aspects of the general theory of moduli of abelian schemes including polarisations and level structures and introduce the respective moduli stacks. In the next section we define Artin-Mazur homotopy types of  Deligne-Mumford stacks and compare the \'etale and classical homotopy types in the algebraic and complex analytic context.  In the third section we determine the \'etale homotopy types of the moduli stacks of abelian schemes over the algebraic closure of the rational numbers. And in the final section we relate the \'etale homotopy types of the moduli stacks of algebraic curves and abelian schemes with principal polarisations by analysing the stacky Torelli morphism. 

\section{Moduli stacks of abelian schemes with polarisations}

We will recollect in this section the main aspects of the theory of abelian schemes with
polarisations and their moduli following Mumford \cite{M1}, Chai \cite{C}, Faltings-Chai \cite{FC} and
Mumford-Fogarthy-Kirwan \cite{MFK}. For the general theory of algebraic stacks we refer to 
the book of Laumon and Moret-Bailly \cite{LMB} and the online Stack Project \cite{SP}. 
Other good sources are \cite{DM}, \cite{Vi}, \cite{Go} and \cite{fa}.

\begin{definition}
Let $S$ be a base scheme. An {\it abelian scheme} $A/S$ is a proper smooth group scheme
$\pi: A\rightarrow S$ over $S$ with connected geometric fibers.
\end{definition}

Being a group scheme, an abelian scheme $A/S$ is equipped with the following structure morphisms:
\begin{itemize}
\item[(1)] a group law or multiplication morphism $\mu: A\times_S A\rightarrow A$,
\item[(2)] a unit section morphism $e: S\rightarrow A$,
\item[(3)] an inverse morphism $i: A\rightarrow A$,
\end{itemize}
subject to the usual axioms for being an abstract group.

It turns out, that in fact any abelian scheme $A/S$ is a commutative group scheme (cf. \cite[Cor. 6.5]{MFK}).\\

Let $A/S$ be an abelian scheme, $\l$ be an invertible sheaf on $A$ and 
$$p_1, p_2: A\times_S A \rightarrow A$$
be the projections to the first and second factor respectively. Then it
follows that the induced sheaf
$$\mu^*\l\otimes p_1^*\l^{-1} \otimes p_2^*\l^{-1}$$
is an invertible sheaf on the fiber product $A\times_S A$.
Regarding $A\times_S A$ as a scheme over $A$ via $p_2$ we get therefore an $S$-morphism
$$\lambda(\l): A \rightarrow \Pic^0(A/S).$$
Here $\Pic^0(A/S)$ is an open subspace and the neutral component of the abelian algebraic space $\Pic(A/S)$ representing the relative Picard functor
$\P\i\c_{A/S}$, which corresponds to the subfunctor  $\P\i\c_{A/S}^0$. $\Pic^0(A/S)$ is in fact an abelian scheme by a theorem of Raynaud \cite[Theorem 1.9]{FC} called the {\it dual abelian scheme} of $A/S$. By the theorem of the cube, the $S$-morphism $\lambda(\l)$ is actually a group homomorphism respecting the group structures of $A/S$ and $\Pic^0(A/S)$ (cf. \cite{MFK}, \cite{FC}, \cite{GeN}). 

Following \cite{MFK} we define the general notion of a polarisation for an abelian scheme:

\begin{definition}
Let $S$ be a base scheme. A {\it polarisation} of an abelian scheme $A/S$ is an $S$-homomorphism of group schemes
$$\lambda: A/S \rightarrow \Pic^0(A/S)$$
such that for each geometric point $s$ of $S$ the induced homomorphism
$$\lambda_s: A_s \rightarrow \Pic^0(A_s)$$
is of the form $\lambda_s=\lambda(\l_s)$ for some ample invertible sheaf $\l_s$ on the fibre $A_s$ of $A$ over the geometric point $s$.

$\lambda$ is called a {\it principal polarisation} if it is an isomorphism.
\end{definition}

Polarisations can be lifted to invertible sheaves globally as follows (cf. \cite{MFK}). Consider the universal (or Poincar\'e) invertible sheaf ${\mathcal P}$ on $A \times_{S} \Pic^0(A/S)$, which is uniquely normalized by imposing that its pull-back on $\Pic^0(A/S)$  via the morphism $$\Pic^0(A/S) \cong S \times_{S} \Pic^0(A/S)  \stackrel{e \times Id} \longrightarrow  A\times_S \Pic^0(A/S)$$ 
is the trivial  bundle ${\mathcal O}_{\Pic^0(A/S)}$. 

Assume that $\lambda$ is a polarisation, consider the map $(Id , \lambda): A \rightarrow A\times_{S} \Pic^0(A/S)$ and denote by ${\mathcal L}_{\lambda} := (Id ,\lambda)^* ({\mathcal P})$. 
We have the morphism  
$\lambda({\mathcal L}_{\lambda}): A \to \Pic^0(A/S)$ defined as above and one can prove that  $\lambda({\mathcal L}_{\lambda}) = 2 \lambda$ and ${\mathcal L}_{\lambda}$ is an ample invertible sheaf (see \cite{MFK} prop. 6.10, \cite{FC} p. 4).

The kernel of $\lambda$ has rank $r^2$ for some locally constant positive integer valued function $r$ on $S$ and $r$ is called the {\it degree} of $\lambda$ (cf. \cite{Ol}, 5.1 and \cite{M2}).  
If $\lambda: A/S \rightarrow \Pic^0(A/S)$ is a polarisation of an abelian scheme $A/S$, then its degree $r$ is constant on each connected component and $\lambda$ is a finite and faithfully flat morphism, i.e an isogeny. In fact, any polarisation of an abelian scheme $A/S$ is a symmetric isogeny
$\lambda: A/S \rightarrow \Pic^0(A/S)$, which locally for the \'etale topology of $S$ is of the form $\lambda(\l)$ for some ample line bundle $\l$ on $A/S$.

If the base scheme $S$ is connected with residue characteristic prime to the degree of $\lambda$, then there exists integers $d_1|d_2|\ldots |d_g$, where $g$ is the relative dimension of $A/S$, such that for any geometric point $s$ of $S$ we have an isomorphism of abelian groups (cf. \cite{GeN})
$$\ker \lambda_s\cong (\Z/d_1\Z\times \Z/d_2\Z\times \ldots \times \Z/d_g\Z)^2.$$
The tuple $D=(d_1, d_2, \ldots d_g)$ is called the {\it type of the polarisation} $\lambda$ of the abelian scheme $A/S$. If $\lambda$ is a principal polarisation, its type is just $D=(1, 1,\ldots 1)$ and especially $d_g=1$.\\

Let $V=\Z^{2g}$ and let $\phi_D: V\times V\rightarrow \Z$ be a symplectic pairing, such that on the standard ordered basis $\{e_i\}_{i=1, \ldots, 2g}$ it holds that
$\phi_D(e_i, e_j)=0$ if $i, j\leq g$ or $i, j\geq g$ and $\phi_D(e_i, e_{j+g})=\delta_{i, j}\cdot d_i$ if  $i, j\leq g$. The group $\CSP_D=\CSP(V, \phi_D)$, called the {\it conformal symplectic group}, is the reductive algebraic group over $\Spec(\Z)$ of symplectic similitudes of $V$ with respect to the form $\phi_D$ and $\nu: \CSP_D\rightarrow \G_m$ is the associated multiplier character. The kernel $\SP_D=\ker \nu$ is the symplectic group $\SP_D$ of all transformations of $V$ preserving the form $\phi_D$. In the special case that $d_g=1$ we get the familiar groups $\CSP_g$ and $\SP_{2g}$ (cf. \cite{MO}).

We will now define the moduli stack of polarised abelian schemes, which will be the main object of our studies.

\begin{definition} Let $g\geq 1$ be a positive integer and $D=(d_1, d_2, \ldots d_g)$ be a tuple of integers with $d_1|d_2|\ldots |d_g$. 
Let $(Sch/\Z[(d_g)^{-1}])$ be the category of schemes over $\Spec(\Z[(d_g)^{-1}])$ and $\a_D$ be the category fibred in groupoids over $(Sch/\Z[(d_g)^{-1}])$ defined by its groupoid of sections $\a_D(S)$ as follows: 

For a scheme $S$ over $\Spec(\Z[(d_g)^{-1}]$, let $\a_{D}(S)$ be the groupoid
whose objects are the pairs $(A/S, \lambda)$, where $A/S$ is an abelian scheme over $S$ and
$$\lambda: A/S \rightarrow \Pic^0(A/S)$$
a polarisation of type $D$. The morphisms of $\a_D(S)$ are $S$-isomorphisms of abelian schemes compatible
with polarisations of type $D$, i.e. a morphism between $(A/S, \lambda)$ and $(A'/S, \lambda')$ is a homomorphism
of abelian schemes $\varphi: A/S \rightarrow A'/S$ such that $\varphi^*(\lambda')= \lambda$.

If we give the category $(Sch/\Z[(d_g)^{-1}])$ the \'etale topology, then $\a_D$ is a stack over the big \'etale site $(Sch/\Z[(d_g)^{-1}])_{et}$, called the {\em moduli stack of abelian schemes with polarisations of type D}.

If $d_g=1$, then $\a_D$ is the {\it moduli stack $\a_g$ of principally polarised abelian schemes of relative dimension $g$} over $(Sch/\Z)_{et}$.
\end{definition}

It turns out that the moduli stacks $\a_D$ behave very nicely. In fact, they are algebraic stacks in the
sense of Deligne-Mumford \cite{DM} on the big \'etale site $(Sch/\Z[(d_g)^{-1}])_{et}$.

\begin{theorem}\label{DMthm}
Let $g\geq 1$ be a positive integer and $D=(d_1, d_2, \ldots d_g)$ be a tuple of positive integers with $d_1|d_2|\ldots |d_g$. 
The moduli stack $\a_D$ is a smooth quasi-projective Deligne-Mumford stack of finite type over $\Spec(\Z[(d_g)^{-1}])$.

In particular for $d_g=1$, the moduli stack $\a_g$ is a quasi-projective smooth Deligne-Mumford stack of finite type over $\Spec(\Z)$ of relative dimension $g(g+1)/2$.
\end{theorem}

\begin{proof}
That $\a_D$ is a Deligne-Mumford stack is a variation of the proof of the same property for the stack $\a_g$ in \cite[4.11]{FC}, \cite[(4.6.3)]{LMB} for moduli of principally polarised abelian schemes using polarisations of general type $D$ and the other assertions follow by general GIT arguments (cf. \cite[Section 2.3]{GeN}).
\end{proof}

We will be mainly interested in the restriction of these moduli stacks $\a_D$ and $\a_g$ to the category of schemes over the field of rational numbers $\Q$ and over its algebraic closure $\bar{\Q}$.

Let us denote by $\a^{an}_D$ the complex analytification or the uniformisation of the moduli stack $\a_D$.
It  is a complex analytic Deligne-Mumford stack, i.e. a complex analytic orbifold over the big site $(AnSp)_{cl}$ of complex analytic spaces with the classical topology of local analytic isomorphisms (cf. \cite{T}, \cite[Section 2.3]{fn} for a general discussion about the complex analytification of algebraic Deligne-Mumford stacks).

In fact, the moduli stack  $\a^{an}_D$ is a complex analytic quotient stack of the form
$$\a^{an}_D=[\HH_g/\SP_D(\Z)]$$
\noindent where $\HH_g$ is the {\it Siegel upper half space of genus $g$} given as

$$\HH_g=\{\Omega\in \Mat(g, \C)|\, \Omega= {}^t\Omega, \IM (\Omega) >0\}$$
and where $\SP_D(\Z) = \{\gamma\in \Mat(2g, \Z) | {}^t\gamma J_D\gamma = J_D\} $ with
$$J_D =\left(
\begin{array}{cc}
0 & D \\
  &  \\
-D & 0\\
\end{array}
\right)
$$ 
Here $D$ is the diagonal matrix with entries $(d_1,...,d_g)$ and the action of $\SP_D(\Z)$ on $\HH_g$  is given as follows: 
\begin{equation}
\label{action}
\left(
\begin{array}{cc}
a & b \\
  &  \\
c & d\\
\end{array}
\right)\Omega = (a\Omega + bD)(D^{-1}c\Omega + D^{-1} d D)^{-1}.
\end{equation}

In fact, using Riemann's bilinear relations (cf. \cite{BL}) one can prove that $\HH_g$ parametrises  isomorphism classes of abelian varieties with a polarisation of type $D$ together with a choice of a symplectic basis. 

More precisely, once we fix $D$, to a point $\Omega \in \HH_g$ we associate the triple $(A_{\Omega} = \C^g/\Lambda_{\Omega}, H_{\Omega}, \{\lambda_1, ..., \lambda_{2g}\})$, where $\Lambda_{\Omega} = \Omega \cdot \Z^g + D\cdot\Z^g$, $H_{\Omega} = (\IM(\Omega))^{-1}$ is the Hermitian form on $\C^g$ which gives the polarisation and $  \{\lambda_1, ..., \lambda_{2g}\}$ are the columns of the $(g\times 2g)$-period matrix $(\Omega \ D)$. One immediately checks that with respect to this basis $\IM(H_{\Omega})_{|\Lambda_{\Omega} \times \Lambda_{\Omega}}$ is given by the matrix $\IM({}^t(\Omega \ D) (\IM(\Omega))^{-1}\overline{(\Omega \ D)} )= J_D$, so the basis  $ \{\lambda_1, ..., \lambda_{2g}\}$  is symplectic. Now, if one wants to parametrise isomorphism classes of abelian varieties with a polarisation of type $D$, one has to take the quotient of $\HH_g$ by the action of the group $\SP_D(\Z)$.
This action is properly discontinuous  (cf. \cite{BL}), hence the quotient $\HH_g/\SP_D(\Z)$ has the structure of a complex analytic orbifold. 

In the case $d_g=1$ we get the complex analytic Deligne-Mumford stack $\a^{an}_g$, which is given as the complex analytic quotient stack or orbifold
$$\a^{an}_g=[\HH_g/\SP(2g, \Z)]$$
where $\SP(2g, \Z)$ is the symplectic group 
$$\SP(2g, \Z)=\{ \gamma\in \Mat(2g, \Z) | {}^t\gamma J\gamma = J\}$$
with
$$J =\left(
\begin{array}{cc}
0 & I_g \\
-I_g & 0\\
\end{array}
\right),
$$
The action \eqref{action}  becomes: 
\begin{eqnarray}
\SP(2g, \Z) \times \HH_g &\rightarrow& \HH_g\nonumber \\
( \left ( \begin{array}{cc}
\vspace*{0.1cm}a & b \\
c & d\\
\end{array}\right ), \Omega ) &\mapsto& (a\Omega + b) \cdot (c\Omega + d)^{-1}. \nonumber
\end{eqnarray}
\noindent and the universal family of principally polarised abelian varieties over the Siegel upper half space $\HH_g$ has as fiber over a point $\Omega \in \HH_g$
$A_{\Omega}=   \C^g/(\Omega \Z^g + \Z^g),$ with the polarisation given by $(\IM(\Omega))^{-1}$. \\

We also like to look at rigidified moduli spaces of abelian schemes using level structures. We recollect here the basic 
definitions and properties (cf. \cite{FC}, \cite{GeN}, \cite{BL}, \cite{C} and \cite{MO}).

As before, let $g, N\geq 1$ be integers and $D=(d_1, d_2, \ldots d_g)$ be a tuple of positive integers with $d_1|d_2|\ldots |d_g$ and $\gcd(d_g, N)=1$. Let $A/S$ be an abelian scheme of relative dimension $g$ over a base scheme $S$ and $$\lambda: A/S \rightarrow \Pic^0(A/S)$$ a polarisation of type $D=(d_1, d_2, \ldots d_g)$. We have the endomorphism of abelian schemes $[N]: A\rightarrow A$ given by multiplication with $N$ and we will denote its kernel by $A[N]=\ker [N]$. $A[N]$ is a locally free finite group scheme of rank $N^{2g}$. We also have the Weil pairing
$$e^D_N: A[N]\times A[N]\rightarrow \mu_N,$$
where $\mu_N$ is the group of $N$-th roots of unity. The morphism $e^D_N$ is a non-degenerate bilinear form, which is also symplectic meaning that $e^D_N (x, x)=1$ for all schemes $U$ over the base $S$ and all $U$-valued points $x\in A[N](U)$. If in addition $N$ is invertible in $\Gamma(S, \O_S)$, the group scheme $A[N]$ is also \'etale over $S$.
Now let $\Z[\zeta_N]=\Z[t]/(\Phi_N)$, where $\Phi_N\in \Z[t]$ is the $N$-th cyclotomic polynomial.

Fix a base scheme $S$ over the ring $\Z[\zeta_N, (Nd_g)^{-1}]$. This means that $S$ is a scheme such that $Nd_g$ is invertible in $\Gamma(S, \O_S)$ and we have fixed a primitive $N$-th root of unity $\zeta_N\in \Gamma(S, \O_S)$. We can interpret $\zeta_N$ as an isomorphism
of group schemes $\zeta_N: (\Z/N\Z)\stackrel{\cong}\longrightarrow A[N]$. 

Let again $V=\Z^{2g}$ and $\phi_D: V\times V\rightarrow \Z$ be a fixed symplectic pairing with the same properties as before. We then define:

\begin{definition} Let $g, N\geq 1$ be integers, $D=(d_1, d_2, \ldots d_g)$ be a tuple of positive integers with $d_1|d_2|\ldots |d_g$ and $\gcd(d_g, N)=1$ and $S$ be a scheme over $\Z[\zeta_N, (Nd_g)^{-1}]$. Let $A/S$ be an abelian scheme of relative dimension $g$ over over $S$ and $$\lambda: A/S \rightarrow \Pic^0(A/S)$$ be a polarisation of type $D=(d_1, d_2, \ldots d_g)$. A {\it (symplectic) level $N$ structure of} $A/S$ is an isomorphism of group schemes
$$\eta: (V/NV)\stackrel{\cong}\longrightarrow A[N]$$ 
such that the diagram
$$\diagram
(V/NV)\times (V/NV) \rto^{\hspace*{1cm}\phi_D} \dto_{\eta\times \eta} &     (\Z/N\Z) \dto^{\zeta_N} \\
A[N]\times A[N] \rto^{\hspace*{1cm}e^D_N} &           \mu_N  \enddiagram $$
is commutative.
\end{definition}

We can now look at the moduli stacks of polarised abelian schemes with level $N$ structures.

\begin{definition} Let $g, N\geq 1$ be integers and $D=(d_1, d_2, \ldots d_g)$ be a tuple of positive integers with $d_1|d_2|\ldots |d_g$ and $\gcd(d_g, N)=1$. Let $(Sch/\Z[\zeta_N, (Nd_g)^{-1}])$ be the category of schemes over $\Z[\zeta_N, (Nd_g)^{-1}]$. Let $\a'_{D, [N]}$ be the category fibred in groupoids over $(Sch/\Z[\zeta_N, (Nd_g)^{-1}])$ defined by its groupoid of sections $\a'_{D, [N]}(S)$ as follows: 

For a scheme $S$ over $\Spec(\Z[\zeta_N, (Nd_g)^{-1}])$, let $\a'_{D, [N]}(S)$ be the groupoid
whose objects are the triples $(A/S, \lambda, \eta)$, where $A/S$ is an abelian scheme over $S$,
$$\lambda: A/S \rightarrow \Pic^0(A/S)$$
a polarisation of type $D$ and 
$$\eta: (V/NV)\stackrel{\cong}\longrightarrow A[N]$$
a level $N$ structure.
The morphisms are $S$-isomorphisms of abelian schemes compatible
with polarisations of type $D$ and level $N$ structures, i.e. a morphism between $(A/S, \lambda, \eta)$ and $(A'/S, \lambda', \eta')$ is a homomorphism
of abelian schemes $\varphi: A/S \rightarrow A'/S$ such that $\varphi^*(\lambda')= \lambda$ and $\eta'=\varphi[N]\circ \eta$, where $\varphi[N]$ is the induced momorphism on the group schemes of $N$-torsion points.

If we give the category $(Sch/\Z[\zeta_N, (Nd_g)^{-1}])$ the \'etale topology, then $\a'_{D, [N]}$ is a stack over the big \'etale site $(Sch/\Z[\zeta_N, (Nd_g)^{-1}])_{et}$,

The {\em moduli stack $\a_{D, [N]}$ of abelian schemes with polarisations of type D and level $N$ structures} is defined as the stack $\a'_{D, [N]}$ viewed as a stack over the big \'etale site $(Sch/\Z[(Nd_g)^{-1}])_{et}$ with structure morphism given by the natural composition
$$\a_{D, [N]}:=\Big( \a'_{D, [N]}\rightarrow \Spec(\Z[\zeta_N, (Nd_g)^{-1}])\rightarrow \Spec(\Z[(Nd_g)^{-1}])\Big).$$

If $d_g=1$, then $\a_{D, [N]}$ is the {\it moduli stack $\a_{g, [N]}$ of principally polarised abelian schemes of relative dimension $g$ with level $N$ structures} over $(Sch/\Z[N^{-1}])_{et}$.
\end{definition}

The moduli stacks of polarised schemes with level structures behave very well. In fact we have the following (cf. \cite{MO}, \cite{FC}).

\begin{theorem}
Let $g, N\geq 1$ be integers and $D=(d_1, d_2, \ldots d_g)$ be a tuple of positive integers with $d_1|d_2|\ldots |d_g$ and $\gcd(d_g, N)=1$. 
The moduli stack $\a'_{D, [N]}$ is a quasi-projective smooth Deligne-Mumford stack of finite type over $\Spec(\Z[\zeta_N, (Nd_g)^{-1}])$ with irreducible geometric fibers. Furthermore, if $N$ is large enough with respect to $D$, then $\a'_{D, [N]}$ is a smooth quasi-projective scheme over $\Spec(\Z[\zeta_N, (Nd_g)^{-1}])$.
The moduli stack $\a_{D, [N]}$ is a quasi-projective smooth Deligne-Mumford stack of finite type over $\Spec(\Z[(Nd_g)^{-1}])$ whose geometric fibers have $\varphi(N)$ irreducible components, where $\varphi$ denotes Euler's function. Furthermore, if $N$ is large enough with respect to $D$, then $\a_{D, [N]}$ is a smooth quasi-projective scheme over $\Spec(\Z[(Nd_g)^{-1}])$. 

In particular for $d_g=1$, the moduli stack $\a_{g, [N]}$ is a quasi-projective smooth Deligne-Mumford stack of finite type over $\Spec(\Z[N^{-1}])$. If in addition $N\geq 3$, then $\a_{g, [N]}$ is a smooth quasi-projective scheme over $\Spec(\Z[N^{-1}])$.
\end{theorem}

\begin{proof} That $\a_{D, [N]}$ is a Deligne-Mumford stack over $\Spec(\Z[(Nd_g)^{-1}])$ follows along the same arguments as for $\a_D$ in  Theorem \ref{DMthm}.

In the special case that $d_g=1$, Serre's lemma \cite{S1} shows that if $N\geq 3$, then the algebraic stack $\a_{g, [N]}$ is a smooth algebraic space over $\Spec(\Z[N^{-1}])$. Mumford proved using GIT methods \cite{MFK} that $\a_{g, [N]}$ is in fact a smooth quasi-projective scheme of finite type over $\Spec(\Z[N^{-1}])$. 
These methods can also be adapted to the general case if $N$ is large enough with respect to $d_g$ (cf. \cite[Theorem 2.3.1]{GeN}).
\end{proof}

Again in the following sections we will be mainly interested in the restriction of these moduli stacks $\a_{D, [N]}$ and $\a_{g, [N]}$ to the category of schemes over the fields $\Q$ and $\bar{\Q}$.\\

Finally, we like to compare the algebraic with the complex analytic setting, which is easier to describe.  So let us denote by $\a^{an}_{D, [N]}$ the complex analytification or the uniformisation of the moduli stack $\a_{D, [N]}$. It  is again a complex analytic Deligne-Mumford stack over the big site $(AnSp)_{cl}$.

In fact, the moduli stack  $\a^{an}_{D, [N]}$ is again given as a quotient stack.  Let us explain this. Consider a complex abelian variety of dimension $g$, $A = W/\Lambda$, where $W$ is a complex vector space of dimension $g$ and $\Lambda$ is a lattice. The set of $N$-torsion points $A[N]$ can be identified with the finite group $N^{-1} \Lambda/\Lambda \cong (\Z/N\Z)^{2g}$. Assume that $E$ is a polarisation of type $D=(d_1,...,d_g)$ with $d_g$ and $N$ coprime. The alternating form $E: \Lambda \times \Lambda \rightarrow \Z$ can be extended to a non-degenerate symplectic form on $\Lambda \otimes \Q$. The Weil pairing $(\alpha, \beta) \mapsto \exp(2\pi i NE(\alpha, \beta))$ defines a symplectic non-degenerate form 
$$e^D_N: A[N] \times A[N] \rightarrow \mu_N, $$
and if we choose a primitive $N$-th root of unity, we can think of the Weil pairing again as a symplectic non-degenerate form with values in $\Z/NZ$. Consider the subgroup $\Gamma_D(N)$ of $\SP_D(\Z)$ consisting of the automorphisms of the pair $(\Lambda, E)$, which induce the trivial action on $\Lambda/N\Lambda$. One can prove that there is a natural bijection between the set of isomorphism classes of polarised abelian varieties of type $D$ equipped with a level $N$ structure and the quotient  $\HH_g/\Gamma_D(N)$. The group $\Gamma_D(N)$ acts properly discontinuously on $\HH_g$, so this quotient has the structure of  a complex analytic orbifold, hence the moduli stack $\a^{an}_{D, [N]}$ can be described as the quotient stack $\a^{an}_{D, [N]}= [\HH_g/\Gamma_D(N)]$.

If $N$ is large enough with respect to the type $D=(d_1, d_2, \ldots ,d_g)$ the group $\SP_D(\Z)$ does not contain torsion elements and acts freely on the Siegel upper half-space $\HH_g$ and we have that
$$\a^{an}_{D, [N]}=\HH_g/\Gamma_D(N)$$
as a quotient being a smooth complex analytic space (cf. \cite{GeN}).\\

Especially again in the case of principally polarised complex abelian varieties, i.e. when $d_g=1$, we get that
$$\a^{an}_{g, [N]}=[\HH_g/\Gamma_{2g}(N)],$$
where the group $\Gamma_{2g}(N)$ is the subgroup of $\SP(2g, \Z)$ corresponding to the symplectic matrices which are congruent to the identity modulo $N$.  And if $N\geq 3$, then we get the quotient
$$\a^{an}_{g, [N]}=\HH_g/\Gamma_{2g}(N),$$
which is a smooth complex analytic space again.

\section{\'Etale and Complex Analytic Homotopy Types for Deligne-Mumford Stacks}

In this section we briefly recall the constructions of \'etale and complex analytic homotopy types for Deligne-Mumford stacks
following the presentation in \cite[Section 2]{fn}, which is based on \cite{AM}, \cite{C} and \cite{Mo}.  

Let $\E$ be a topos with an initial object $\emptyset$. An object $X$ of $\E$ is {\it connected},
if whenever $X=X_1\coprod X_2$, either $X_1$ or $X_2$ is the initial object $\emptyset$.
A topos $\E$ is {\it locally connected} if every object of $\E$ is a coproduct of connected objects.
For a locally connected topos $\E$ define the {\it connected component functor} $\pi$ as the functor
$$\pi: \E \rightarrow (Sets)$$ associating to any object $X$ of $\E$ its set $\pi(X)$ of connected components.

Let $\Delta$ be the category of simplices, i.e. the category whose objects are sets
$[n]=\{0, 1, 2, \ldots n\}$ and whose morphisms are non-decreasing maps. We denote by $\Delta_n$ the
full subcategory of $\Delta$ of all sets $[k]$ with $k\leq n$. A simplicial object in a topos $\E$ is a functor $X_{\bullet}: \Delta^{op} \rightarrow \E$ denoted by $X_{\bullet}=\{X_n\}_{n\geq 0}$.   We will write $\Delta^{op} \E$ for the category of all simplicial objects in $\E$.
The restriction or $n$-truncation functor
$$(-)^{(n)}: \Delta^{op} \E \rightarrow \Delta_n^{op} \E$$
has left and right adjoint functors, the {\it $n$-th skeleton} and {\it $n$-th coskeleton functors}
$$\sk_n, \; \cosk_n: \Delta_n^{op} \E \rightarrow \Delta^{op} \E.$$

\begin{definition}
A {\it hypercovering} of a topos $\E$ is a simplicial object $U_{\bullet}=\{U_n\}_{n\geq 0}$ in $\E$ such that the
morphisms
$$U_0 \rightarrow *$$
$$U_{n+1}\rightarrow \cosk_n(U_{\bullet})_{n+1}$$
are epimorphisms, where $*$ is the final object of $\E$.
\end{definition}

If $S$ is a set and $X$ an object of a topos $\E$ define $S\otimes X:= \coprod_{s\in S}X$.
If $S_{\bullet}$ is a simplicial set and $X_{\bullet}$ a simplicial object of $\E$, define
$$S_{\bullet}\otimes X_{\bullet}: \Delta^{op} \rightarrow \E$$
to be the simplicial object given by $(S_{\bullet}\otimes X_{\bullet})_n:=S_n \otimes X_n.$

Let $\Delta[m] =\Hom_{\Delta}(-, [m])$ be the standard $m$-simplex, i.e. the functor $\Delta[m]: \Delta^{op}
\rightarrow (Sets)$ represented by the object $[m]$.

Two morphisms $f,g: X_{\bullet}\rightarrow Y_{\bullet}$ are called {\it strictly homotopic} if there is a
morphism $H: X_{\bullet} \otimes \Delta[1]\rightarrow Y_{\bullet}$, called {\it strict homotopy}, such that the following diagram
is commutative
$$\diagram
X_{\bullet}=X_{\bullet}\otimes\Delta[0] \dto_{id\otimes d^0} \drrto^{f} &   &  \\
X_{\bullet}\otimes\Delta[1] \rrto^{H}                                        &   & Y_{\bullet} \\
X_{\bullet}=X_{\bullet}\otimes\Delta[0] \uto^{id\otimes d^1} \urrto^{g} &   &   \enddiagram$$

Two morphisms $f,g: X_{\bullet}\rightarrow Y_{\bullet}$ are called {\it homotopic}, if they can be related by a chain
of strict homotopies. Homotopy is the equivalence relation generated by strict homotopy. 

Let $\h\r (\E)$ be the {\it homotopy category of hypercoverings of} $\E$, i.e., the category of hypercoverings of $\E$ with their morphisms up to homotopy.
It can be shown that its opposite category $\h\r (\E)^{op}$ is in fact a filtering category (cf. \cite{AM}).

Let $\h=\h(\Delta^{op}(Sets))$ denote the {\it homotopy category of simplicial sets}, i.e., the category of simplicial sets with their morphisms up to homotopy. 
It turns out that this category is actually equivalent to the classical homotopy category of CW-complexes and their maps up to homotopy (cf. \cite{GJ}, Chap. I or \cite{BK}, VIII.3). 

In addition let $\pro-\h$ denote the category of pro-objects in the category $\h$, i.e. the category of (contravariant) functors
$X: \I \rightarrow \h$ from a filtering index category $\I$ to $\h$. We will write normally $X=\{X_i\}_{i\in \I}$ to
indicate that we think of pro-objects $X$ as inverse systems of objects of $\h$ (cf. \cite[Appendix]{AM}).

Let $\E$ be a locally connected topos. The connected component functor $\pi: \E \rightarrow (Sets)$ induces a functor
$$\Delta^{op}\pi: \Delta^{op}\E \rightarrow \Delta^{op}(Sets).$$
\noindent If we pass to homotopy categories and restrict to hypercoverings of $\E$ we obtain a functor
$$\pi: \h\r(\E) \rightarrow \pro-\h.$$

Finally we can define the Artin-Mazur homotopy type of a locally connected topos, following \cite[\S 9]{AM}, \cite{Mo}.

\begin{definition}
Let $\E$ be a locally connected topos. The {\it Artin-Mazur homotopy type} of $\E$ is the pro-object
$$\{\E\}= \{\pi(U_{\bullet})\}_{U_{\bullet}\in \h\r(\E)}.$$
in the homotopy category of simplicial sets. The Artin-Mazur homotopy type is functorial with respect to morphisms of topoi, the associated functor
$\{-\}$ is called the {\it Verdier functor}.
\end{definition}
If $\E$ is a locally connected topos and $x$ a point of the topos $\E$, i.e. a morphism of topoi
$x: (Sets) \rightarrow \E$,
we can also define the {\it homotopy pro-groups} $\pi_n(\E, x)$ of $\E$ for $n\geq 1$.
(cf. \cite{AM}, \cite{Mo} and \cite{Z}). 

Now we will recall the definition of the Artin-Mazur \'etale homotopy type for the topos of sheaves on the small \'etale site
of an algebraic Deligne-Mumford stack (cf. \cite[Section 2.2]{fn}). For the general theory of algebraic stacks we refer again 
to \cite{LMB} and \cite{SP}. 

Let $\s$ be an algebraic Deligne-Mumford stack over the big \'etale site of the category of schemes $(Sch/B)$ over a fixed base scheme $B$. 
Then we can consider its small \'etale site $\s _{et}$ (cf. \cite{LMB}, \cite{fn}). The objects are \'etale $1$-morphisms $x: X\rightarrow \s$,
where $X$ is a scheme, morphisms are morphisms over $\s$, i.e. $2$-commutative
diagrams of the form
$$\diagram
   X \rrto^f \drto_x&               & Y\dlto^y \\
                 & \s &                 \enddiagram $$                
\noindent with $f: X\rightarrow Y$ a $1$-morphism, together with a $2$-isomorphism $\alpha: x\Rightarrow y\circ f$. The coverings are the \'etale coverings of the schemes involved.

We let $\E_{et}=\hf\vf(\s_{et})$ be the associated \'etale topos, i.e. the category of
sheaves on the small \'etale site $\s _{et}$ of the algebraic stack $\s$. It turns out that this topos $\E_{et}$ is
in fact a locally connected topos (cf. \cite[3.1]{Z}).

We define the \'etale homotopy type of an algebraic Deligne-Mumford stack as the Artin-Mazur homotopy
type of its \'etale topos (cf. \cite{fn}). 

\begin{definition}
Let $\s$ be an algebraic Deligne-Mumford stack. The {\it \'etale homotopy type} $\{\s\}_{et}$ of $\s$ is the
pro-object in the homotopy category of simplicial sets given as
$$\{\s\}_{et}=\{\E_{et}\}.$$
The {\it \'etale homotopy groups} of the stack $\s$ are defined as the Artin-Mazur homotopy pro-groups $\pi_n^{et}(\s, x)=\pi_n(\E_{et}, x)$
of the topos $\E_{et}$.
\end{definition}

For $n=1$ this gives the \'etale fundamental group $\pi_1^{et}(\s, x)$ of the algebraic stack $\s$, which is discussed in 
detail also in Noohi \cite{No} and Zoonekynd \cite{Z}.

As was shown in \cite[Section 2.2]{fn} we can determine the \'etale homotopy type of a Deligne-Mumford stack $\s$ directly from that of a hypercovering $X_{\bullet}$ of the \'etale topos $\E_{et}$, i.e. from a simplicial scheme over $\s$ which can often be constructed directly for example from a given \'etale atlas of the Deligne-Mumford stack. This is done using the following generalisation of the homotopy descent theorem for simplicial schemes of Cox \cite[Theorem IV.2]{C} (cf. also \cite[Theorem 3]{O}). The \'etale homotopy type $\{X_{\bullet}\e$ of the simplicial scheme $X_{\bullet}$ is defined as the Artin-Mazur pro-homotopy type of the locally connected \'etale topos $\hf\vf(X_{\bullet})_{et}$ of \'etale sheaves on $X_{\bullet}$, i.e. $\{X_{\bullet}\e= \{\hf\vf(X_{\bullet})_{et}\}$ (cf. Cox \cite[Chap. I \& III]{C} and \cite{F}).

We have the following stacky interpretation \cite[Theorem 2.3]{fn} (cf. also \cite[Theorem 3]{O}) of  the homotopy descent theorem for simplicial schemes of Cox \cite[Theorem IV.2]{C}. 

\begin{theorem}
\label{descthm}
If $\s$ is an algebraic Deligne-Mumford stack with small \'etale site $\s_{et}$ and
$X_{\bullet}$ a simplicial scheme which is a hypercovering of the \'etale topos $\E_{et}$,
then there is a weak homotopy equivalence of pro-simplicial sets
$$\{\s\e \simeq \{X_{\bullet}\e.$$
where $\{X_{\bullet}\e$ is the \'etale homotopy type of the simplicial scheme $X_{\bullet}$.
\end{theorem}

Therefore the determination of the \'etale homotopy type $\{\s\e$ of an algebraic Deligne-Mumford stack $\s$ is reduced to the determination of the \'etale homotopy type of a simplicial scheme $X_{\bullet}$ being any hypercovering of the \'etale topos $\E_{et}$ and this does not depend on the particular hypercovering (cf. \cite{C}, \cite{O}, \cite{fn}).

Cox \cite[\S 4]{C} also showed that there is an alternative and more direct way to define the \'etale homotopy type of a simplicial scheme $\{X_{\bullet}\}_{et}$ and therefore to determine the \'etale homotopy type of an algebraic Deligne-Mumford stack. Namely, let $U_{\bullet \bullet}$ be a hypercovering of the \'etale topos $\hf\vf(X_{\bullet})_{et}$ of the simplicial scheme $X_{\bullet}$, which is a bisimplicial object $U_{\bullet \bullet}=\{U_{m n}\}_{m, n}$. Applying the connected component functor to each of the $U_{m n}$ gives a bisimplicial object $\pi(U_{\bullet \bullet})$ and taking its diagonal we get a simplicial set $\Delta \pi(U_{\bullet \bullet})$. Running through all hypercoverings
$U_{\bullet \bullet}\in \h\r(\hf\vf(X_{\bullet})_{et})$ gives a pro-object $\{\Delta \pi(U_{\bullet \bullet})\}\in \pro-\h$ in the homotopy category of simplicial sets. Cox \cite[Theorem III.8]{C} showed that there is a weak equivalence
$$\{X_{\bullet}\}_{et}\simeq \{\Delta \pi(U_{\bullet \bullet})\}$$
in $\pro-\h$.

Similarly we can consider Artin-Mazur pro-homotopy types of analytic Deligne-Mumford stacks in the context of complex analytic spaces (cf. \cite[Section 2.3]{fn}). We also refer to \cite[Chap. 5]{T} for the general theory and properties of complex analytic stacks and analytifications of algebraic stacks.

Let $\t$ be an analytic Deligne-Mumford stack over the big site of local analytic isomorphisms of the category of complex analytic spaces (AnSp/C) over a base complex analytic space $C$. Local analytic isomorphism between complex analytic spaces here play the same role as \'etale maps between schemes (cf. \cite{Gr}).

For a given analytic Deligne-Mumford stack $\t$ we can define its small site of local analytic isomorphisms $\t_{cl}$ given by local analytic isomorphisms $x: X\rightarrow \t$, where $X$ is a complex analytic space and morphisms are morphisms of complex analytic spaces over
$\t$ and the coverings are given by families of local analytic isomorphisms (cf. \cite{M},
\cite[Chap. IV, \S3]{C}). 
Let $\E_{cl}=\hf\vf(\t_{cl})$ be the associated topos of local analytic isomorphisms, i.e. the category of
sheaves on the small site of local analytic isomorphisms of the analytic Deligne-Mumford stack $\t$. It is again a locally connected topos
and we can define the complex analytic homotopy type of $\t$ in the sense of Artin-Mazur formally as before.

\begin{definition}
Let $\t$ be an analytic Deligne-Mumford stack. The {\it complex analytic} or {\it classical homotopy type} $\{\t\}_{cl}$ of $\t$ is the
pro-object in the homotopy category of simplicial sets given as
$$\{\t\}_{cl}=\{\E_{cl}\}.$$
The {\it classical homotopy groups} of the stack $\t$ are defined as the Artin-Mazur homotopy pro-groups $\pi_n^{cl}(\t, x)=\pi_n(\E_{cl}, x)$
of the topos $\E_{cl}$.
\end{definition}

Also in the complex analytic setting we have an explicit description of the classical homotopy type $\{\t\}_{cl}$ of an analytic Deligne-Mumford stack using hypercoverings of the topos $\E_{cl}$. Namely, if $\t$ is an analytic Deligne-Mumford stack and $X_{\bullet}$ is a simplicial analytic
space, which is also a hypercovering of the topos $\E_{cl}$, then homotopy descent gives a weak homotopy equivalence of pro-simplicial sets
$$\{\t\}_{cl} \simeq \{X_{\bullet}\}_{cl}$$
\noindent where $\{X_{\bullet}\}_{cl}$ is the classical pro-homotopy type of the simplicial complex analytic space $X_{\bullet}$ (cf. \cite{fn}, \cite[\S 3]{C}. But in the case of complex analytic spaces it turns out that this pro-homotopy type $\{X_{\bullet}\}_{cl}$ is in fact just an ordinary homotopy type as we can see as follows.
Following the same arguments as in the algebraic setting, Cox \cite[\S 3]{C} also showed that there is a weak equivalence 
$$\{X_{\bullet}\}_{cl}\simeq \{\Delta \pi(U_{\bullet \bullet})\}$$
in $\pro-\h$, where $U_{\bullet \bullet}$ is a hypercovering of the topos $\hf\vf(X_{\bullet})_{cl}$ of the simplicial complex analytic space $X_{\bullet}$.
But Cox \cite[Proposition IV.9]{C} then proves that in the complex analytic situation $\{\Delta \pi(U_{\bullet \bullet})\}$ gives actually an ordinary homotopy type, namely there is a weak equivalence 
$$\{\Delta \pi(U_{\bullet \bullet})\}\simeq \Delta \Sin (X_{\bullet})$$
in $\pro-\h$, where $\Delta \Sin (X_{\bullet})$ is the simplicial set given by taking the diagonal $\Delta$ of the bisimplicial set $\Sin (X_{\bullet})$ 
given in bidegree $s,t$ by $\Sin (X_s)_t$ and where $\Sin$ is the singular functor (cf. \cite[Chap. IV, \S3]{C} and \cite[Chap. 8]{F}). Furthermore there is a canonical weak equivalence
$$\Delta \Sin (X_{\bullet})\stackrel{\simeq}\longrightarrow \Sin (|X_{\bullet}|),$$
where $|X_{\bullet}|$ denotes the geometric realisation of the simplicial analytic space $X_{\bullet}$. From this we get finally a weak equivalence
$$\{\t\}_{cl} \simeq \Sin (|X_{\bullet}|).$$
\noindent So the classical homotopy type of an analytic Deligne-Mumford stack $\t$ is given as the ordinary homotopy type of any hypercovering. Again this construction is independent of the actual choice of the hypercovering $X_{\bullet}$ (cf. \cite{C}, \cite{O}, \cite{fn}).\\

Now we fix an embedding $\bar{\Q}\hookrightarrow \C$ of the algebraic closure of the rational numbers in the complex numbers. 
For a given algebraic Deligne-Mumford stack $\s$ over $\Spec(\bar{\Q})$
let $\s^{an}$ be the associated analytic Deligne-Mumford stack (cf. \cite[Chap. 5]{T}).
Similarly for any scheme $X$ over $\Spec(\bar{\Q})$, let $X^{an}$ denote the complex
analytic space associated with the $\C$-valued points $X(\C)$ of the scheme $X$.

Both the \'etale and the classical homotopy type are determined by a hypercovering of their respective topoi 
and we recall the following general comparison theorem between \'etale and classical homotopy types of Deligne-Mumford stacks
\cite[Theorem 2.4]{fn}, which is a consequence of the comparison theorem for homotopy types of simplicial schemes
by Cox \cite[Theorem IV.8]{C} and Friedlander \cite[Theorem 8.4]{F}.

\begin{theorem}
\label{homdesc}
Let $\s$ be an algebraic Deligne-Mumford stack over $\bar{\Q}$ and
$X_{\bullet}$ a simplicial scheme of finite type over $\bar{\Q}$ which is a hypercovering
of the topos $\E_{et}=\hf\vf(\s_{et})$. If $X^{an}_{\bullet}$ is the associated simplicial
complex analytic space of $X_{\bullet}$, then there is a weak homotopy equivalence of
pro-simplicial sets
$$\{\s\}^{\wedge}_{et}\simeq \Sin (|X^{an}_{\bullet}|)^{\wedge}.$$
\noindent where $^{\wedge}$ denotes the Artin-Mazur profinite completion functor on the homotopy category 
of simplicial sets.
\end{theorem}

This comparison theorem now allows to determine the \'etale homotopy types over $\bar{\Q}$ of the moduli stacks of polarised abelian schemes we introduced in the last section. 

\section{\'Etale homotopy types of moduli stacks of polarised abelian schemes}

We will now prove our main theorem on the \'etale homotopy type of the moduli stacks of polarised abelian schemes using the general comparison theorem for \'etale homotopy types as outlined in the last section.

Let $g\geq 1$ be a positive integer and $D=(d_1, d_2, \ldots d_g)$ be a tuple of integers with $d_1|d_2|\ldots |d_g$. We want to study the \'etale homotopy type of the moduli stack $\a_D$ of abelian schemes with polarisations of type $D$. 

We denote by ${\a}_{D} \otimes \bar{\Q}$ the restriction of the moduli stack ${\a}_{D}$ to the category of schemes over $\bar{\Q}$, i.e. the extension of scalars of $\a_D$ to the algebraic closure  $\bar{\Q}$ of the field $\Q$ of rational numbers.

We also fix an embedding $\bar{\Q} \hookrightarrow {\C}$ of the algebraic closure of the rationals in the complex numbers.
The analytic stack ${\a}_{D}^{an}$ is precisely the complex analytification 
$({\a}_{D}\otimes\bar{\Q})^{an}$ of the Deligne-Mumford stack ${\a}_{D}\otimes\bar{\Q}$ and
we can now determine the \'etale homotopy type of ${\a}_{D}\otimes\bar{\Q}$.

Let us first determine the classical homotopy type of the complex analytification $\a_D^{an}$, which as we have discussed in the last section is given by an ordinary homotopy type of a simplicial set. 

\begin{theorem}
\label{anaho} 
Let $g\geq 1$ be a positive integer and $D=(d_1, d_2, \ldots d_g)$ be a tuple of integers with $d_1|d_2|\ldots |d_g$. 
There is a weak homotopy equivalence of simplicial sets 
$$\{\a_D^{an}\}_{cl} \simeq \Sin(|\BB\SP_D(\Z)|).$$
In particular for principal polarisations we have that
$$\{\a_g^{an}\}_{cl} \simeq \Sin(|\BB\SP(2g, \Z)|).$$
\end{theorem}

\begin{proof}
Since the moduli stack $\a_D^{an}$ is given as the quotient stack 
$$\a^{an}_D=[\HH_g/\SP_D(\Z)],$$ we have the following cartesian diagram
$$\diagram
\Isom(\pi, \pi)= \HH_g \times_{{\a}_{D}^{an}} \HH_g \rto \dto &     \HH_g \dto^{\pi} \\
\HH_g \rto^{\pi} &            {\a}_{D}^{an}  \enddiagram $$ 
where the morphism $\pi$ is an atlas, which is a local analytic isomorphism and both projections from $\Isom(\pi, \pi)= \HH_g \times_{{\a}_{D}^{an}} \HH_g$ are local analytic isomorphisms. The atlas $\pi: \HH_g \rightarrow  {\a}_{D}^{an}$ corresponds to the trivial principal $\SP_D(\Z)$-bundle $\HH_g \times \SP_D(\Z)$ together with the $\SP_D(\Z)$-equivariant morphism given by the right action $\rho: \SP_D(\Z) \times \HH_g  \rightarrow \HH_g$ (cf. \cite{DM}, \cite{Go}). Therefore we have 
$$\Isom(\pi, \pi)= \HH_g \times_{{\a}_{D}^{an}} \HH_g  \cong \HH_g \times \SP_D(\Z).$$
By induction on the number of factors, we get an isomorphism
$$\HH_g \times_{{\a}_{D}^{an}}\HH_g\hdots \times_{{\a}_{D}^{an}} \HH_g \cong \HH_g \times \SP_D(\Z) \times \hdots \times \SP_D(\Z).$$
This allows us to determine the classical homotopy type of the uniformisation ${\a}_{D}^{an}$. 
Let $\cosk_0^{{\a}_{D}^{an}}(\HH_g)_{\bullet}$ be the relative \v{C}ech nerve. It is a hypercovering associated to the atlas $\pi: \HH_g\rightarrow {\a}_{D}^{an}.$
The induction above shows that, the $m$-simplex $\cosk_0^{{\a}_{D}^{an}}(\HH_g)_m$ of the
\v{C}ech nerve $\cosk_0^{{\a}_{D}^{an}}(\HH_g)_{\bullet}$ is given by the $(m +1)$-tuple
fiber product of copies of $\HH_g$ over ${\a}_{D}^{an}$, so we have:
$$\cosk_0^{{\a}_{D}^{an}}(\HH_g)_m \cong \HH_g \times \SP_D(\Z) \times \SP_D(\Z) \times \cdots \SP_D(\Z).$$
And by definition, we get
$$\cosk_0^{{\a}_{D}^{an}}(\HH_g)_{\bullet}\cong \HH_g \times \Nn(\SP_D(\Z))_{\bullet}$$
where the simplicial set $\Nn(\SP_D(\Z))_{\bullet}$ is the simplicial nerve of the group $\SP_D(\Z)$, viewed as a category with one object and with $\SP_D(\Z)$ as the set of morphisms. The geometric realisation of  $\Nn(\SP_D(\Z))_{\bullet}$ is the classifying space $\BB\SP_D(\Z)$. Therefore after geometric realisation we have a homeomorphism of topological spaces
$$|\cosk_0^{{\a}_{D}^{an}}(\HH_g)_{\bullet}|\cong \HH_g \times \BB\SP_D(\Z).$$

The Siegel upper half  space $\HH_g=\SP(2g, \R)/U(g)$ is a homogeneous space given as the quotient $\SP(2g, \R)/U(g)$, where the unitary group $U(g)$ is a maximal compact subgroup of $\SP(2g, \R)$. Moreover the Siegel upper half space $\HH_g$ is contractible (cf. \cite[2. 20]{McDS}). Therefore we finally conclude that there is a homotopy equivalence
$$|\cosk_0^{{\a}_{D}^{an}}(\HH_g)_{\bullet}| \simeq\BB\SP_D(\Z)$$
which determines the classical homotopy type of the uniformisation $\a^{an}_D$ of the moduli stack $\a_D.$

In the case that the polarisations are principal, i.e. $d_g=1$ we have especially
$$|\cosk_0^{{\a}_{g}^{an}}(\HH_g)_{\bullet}| \simeq\BB\SP(2g, \Z)$$
determining the classical homotopy of the uniformisation $\a^{an}_g$ of the moduli stack $\a_g$ of principally polarised abelian schemes.

As the referee pointed out we could also argue in a more systematic way as follows: Let $G$ be a discrete group and assume we have a $G$-equivariant weak equivalence $f: X\stackrel{\simeq}\rightarrow Y$ of topological spaces. Let $f_{\bullet}: X_{\bullet}\rightarrow Y_{\bullet}$ be the induced map of simplicial sets, where $X_{\bullet}=\{X_n\}_n$ with $X_n=X\times G\times G\times\cdots\times G$ and $Y_{\bullet}=\{Y_n\}_n$ with $Y_n=Y\times G\times G\times\cdots\times G$ both featuring $n$ copies of $G$. Then it follows that for each $n$ the map $f_n: X_n\stackrel{\simeq}\rightarrow Y_n$ is also a weak equivalence and therefore that the induced map $\Delta\Sin(X_{\bullet})\stackrel{\simeq}\rightarrow \Delta\Sin(Y_{\bullet})$ is a weak equivalence of simplicial sets (cf. \cite[p. 94]{Q}, \cite[Chap. IV, Prop. 1.7]{GJ}, \cite[Theorem B.2]{BF} or \cite{BK}). So finally we end up with the weak equivalence $\Sin (|X_{\bullet}|)\simeq \Sin (|Y_{\bullet}|)$ of simplicial sets using the considerations of the previous section on homotopy types. Now applying this to the special case where $G=\SP_D(\Z)$, $X=\HH_g$ and $Y$ being a point gives weak equivalences
$$\{\a_g^{an}\}_{cl} \simeq \Sin(|(\HH_g)_{\bullet}|)\simeq \Sin(|\BB\SP_D(\Z)|)$$
of simplicial sets.
\end{proof}

Now we can determine the \'etale homotopy types of the moduli stacks $\a_D$ using the comparison theorem (Theorem \ref{homdesc}).

\begin{theorem}
\label{etho}
Let $g\geq 1$ be a positive integer and $D=(d_1, d_2, \ldots d_g)$ be a tuple of integers with $d_1|d_2|\ldots |d_g$. 
There is a weak homotopy equivalence of pro-simplicial sets
$$\{\a_D \otimes\bar{\Q} \}^{\wedge}_{et} \simeq K(\SP_D (\Z), 1)^{\wedge}.$$
where $^{\wedge}$ denotes Artin-Mazur profinite completion.
In particular for principal polarisations we have that
$$\{\a_g \otimes\bar{\Q} \}^{\wedge}_{et} \simeq K(\SP (2g, \Z), 1)^{\wedge}.$$
\end{theorem}

\begin{proof} The moduli stack ${\a}_{g}\otimes\bar{\Q}$ is a Deligne-Mumford stack, so there exists a 
scheme $X$ over $\bar{\Q}$ and an \'etale surjective morphism
$x: X\rightarrow {\a}_{D}\otimes\bar{\Q}.$
The relative \v{C}ech nerve $\cosk_0^{{\a}_{D}\otimes\bar{\Q}}(X)_{\bullet}$ for
this morphism defines a hypercovering of the stack ${\a}_{D}\otimes\bar{\Q}$.
Similarly, $\cosk_0^{({\a}_{D}\otimes\bar{\Q})^{an}}(X^{an})_{\bullet}$ is a hypercovering of the complex analytic
stack ${\a}_{D}^{an}$. Here $X^{an}$ denotes the associated complex analytic
space of the covering scheme $X$ over $\bar{\Q}$.

The homotopy descent theorem (Theorem \ref{descthm}) now implies that there is a weak equivalences of pro-simplicial sets:
$$\{\a_D\otimes\bar{\Q}\e \simeq \{\cosk_0^{\a_{D}\otimes\bar{\Q}}(X)_{\bullet}\e.$$

The \'etale homotopy type does not change under base change for algebraically closed fields in characteristic zero (cf. \cite[Cor. 12.12]{AM}) and we have:
$$\{\cosk_0^{\a_{D}\otimes\bar{\Q}}(X)_{\bullet}\}^\wedge_{et} \simeq 
\{\cosk_0^{\a_{D}\otimes\bar{\Q}}(X)_{\bullet}\otimes_{\bar{\Q}}\C\}^\wedge_{et}.$$
The comparison theorem for simplicial schemes (cf. Cox \cite[Thm. IV.8]{C}) then implies that we get in fact a weak
equivalence after profinite completions:
$$\{\cosk_0^{\a_{D}\otimes\bar{\Q}}(X)_{\bullet}\otimes_{\bar{\Q}}\C\}^{\wedge}_{et} \simeq 
\Sin(|\cosk_0^{({\a}_{D}\otimes\bar{\Q})^{an}}(X^{an})_{\bullet}|)^\wedge.$$
Because $\cosk_0^{({\a}_{D}\otimes\bar{\Q})^{an}}(X^{an})_{\bullet}$ is a hypercovering of the complex analytic
stack ${\a}_D^{an}$ we have a weak equivalence:
$$\{\a_{D}^{an}\}_{cl} \simeq \Sin(|\cosk_0^{({\a}_{D}\otimes\bar{\Q})^{an}}(X^{an})_{\bullet}|)$$
From the determination of the classical homotopy type of ${\a}_{D}^{an}$ in Theorem \ref{anaho} we get therefore
finally the following weak homotopy equivalences:
$$\{{\a}_{D}\otimes\bar{\Q}\}^{\wedge}_{et}\simeq \{\a_{D}^{an}\}^{\wedge}_{cl}\simeq
 \BB\SP_D(\Z)^{\wedge} =K(\SP_D(\Z), 1)^{\wedge}.$$
The particular case for the moduli stack $\a_g$ follows now immediately from this as $d_g=1.$
\end{proof}

Let us remark that the group $\SP(2g, \Z)$ for $g\geq 2$ is not a good group in the sense of Serre  (cf. \cite[Lemma 3.16]{Mk}, \cite{S}, \cite[Definition 6.5]{AM} and \cite{GJZZ}), which implies, by the theorem of Artin-Mazur \cite[6.6]{AM}, that  the profinite completion $K(\SP(2g, \Z), 1)^{\wedge}$ is {\em not} weakly equivalent to $K(\widehat{\SP(2g, \Z)}, 1)$. The group $\SL(2, \Z)$ though is a good group as it contains a free subgroup of finite index (cf. \cite{S}). On the other hand it is not known in general if the mapping class group $\Map_{g, n}$ of  compact Riemann surfaces of genus $g$ with $n$ punctures is a good group (cf. \cite[3.4]{LS}).

As an immediate consequence we can also determine the \'etale homotopy type of the moduli stack $\m_{ell}$ of elliptic curves 
\cite{DR}, which is simply given as the moduli stack $\a_1$ of principally polarised abelian schemes. In this particular case we have a complex analytic orbifold $\m_{ell}^{an}=[\HH_1/SL(2, \Z)]$, where $\HH_1$ is now just the complex upper half plane. And so with the same arguments as in the proof of Theorem \ref{etho} and the previous remark we get:

\begin{corollary}\label{ellip}
There are weak homotopy equivalences of pro-simplicial sets
$$\{\m_{ell} \otimes\bar{\Q} \}^{\wedge}_{et} \simeq K(\SL (2, \Z), 1)^{\wedge}\simeq K(\widehat{\SL(2, \Z)}, 1).$$
where $\SL (2, \Z)$ is the integral special linear group and $^{\wedge}$ denotes
Artin-Mazur profinite completion.
\end{corollary}

Finally we can also interpret Theorem \ref{etho} and Corollary \ref{ellip} in terms of \'etale fundamental groups of algebraic stacks and get the following:

\begin{corollary}
Let $x$ be any point in the moduli stack ${\a}_{D}\otimes\bar{\Q}$, then we have for the \'etale fundamental group
$$\pi_1^{et}({\a}_{D}\otimes\bar{\Q}, x)\cong \SP_D(\Z)^{\wedge}.$$
Especially for principally polarised schemes we have that
$$\pi_1^{et}({\a}_{g}\otimes\bar{\Q}, x)  =  \prod_{p \ prime}\SP(2g, \Z_p).$$
\end{corollary}

\begin{proof}
By Theorem \ref{etho} and by \cite[3.7]{AM}, we conclude that 
$$\pi_1^{et}({\a}_{D}\otimes\bar{\Q}, x)  \cong \SP_D(\Z)^{\wedge}.$$
The affirmative solution of the congruence subgroup problem for the discrete group $\SP(2g, \Z)$ (cf. \cite{BMS}) in addition implies that (cf. also \cite[Lemma 3.16]{Mk})
$$\SP(2g, \Z)^{\wedge} = \SP(2g, \widehat{\Z}) \cong  \prod_{p \ prime}\SP(2g, \Z_p).$$
and so we get the desired description of the \'etale fundamental group for the moduli stack $\a_g$ of principally polarised abelian schemes.
\end{proof}

For the moduli stack of elliptic curves it follows in particular

\begin{corollary}
Let $x$ be a point in the moduli stack $\m_{ell} \otimes\bar{\Q}$, then we have for the \'etale fundamental group
$$\pi_1^{et}(\m_{ell} \otimes\bar{\Q}, x)\cong \SL(2, \Z)^{\wedge}$$
and for the higher \'etale homotopy groups 
$$\pi_n^{et}(\m_{ell} \otimes\bar{\Q}, x)=0$$
for all $n\geq 2$.
\end{corollary}

The Grothendieck exact sequence of \'etale fundamental groups now allows us to relate the \'etale fundamental groups of the moduli stacks ${\a}_{g}\otimes \Q$ with the profinite completions $\SP(2g, \widehat{\Z})$ and with the absolute Galois group $\Gal(\bar{\Q}/\Q)$. 

\begin{corollary}
Let $x$ be any point in the moduli stack ${\a}_{g}\otimes\bar{\Q}$, then there is a short exact sequence of profinite groups
$$1\rightarrow \prod_{p \ prime}\SP(2g, \Z_p)  \rightarrow \pi^{et}_{1}({\a}_g \otimes \Q, x) \rightarrow \Gal(\bar{\Q}/\Q)\rightarrow 1.$$
\end{corollary}

To finish this section we determine the \'etale homotopy types of the moduli stacks of polarised abelian schemes with level structures. So let again $g, N\geq 1$ be integers and $D=(d_1, d_2, \ldots d_g)$ be a tuple of positive integers with $d_1|d_2|\ldots |d_g$ and additionally $\gcd(d_g, N)=1$.

First, for the complex analytic stack $\a_{D, [N]}$ we get the following

\begin{theorem}
\label{anahoN} Let $g, N\geq 1$ be integers and $D=(d_1, d_2, \ldots d_g)$ be a tuple of positive integers with $d_1|d_2|\ldots |d_g$ and  $\gcd(d_g, N)=1$. There is a weak homotopy equivalence of simplicial sets 
$$\{\a_{D, [N]}^{an}\}_{cl} \simeq \Sin(|\BB\Gamma_D(N)|).$$
For principal polarisations with level $N$ structures we have in particular that
$$\{\a_{g, [N]}^{an}\}_{cl} \simeq \Sin(|\BB\Gamma_{2g}(N)|).$$
\end{theorem}

\begin{proof} This follows as in the proof of Theorem \ref{anaho} because the moduli stack $\a_{D, [N]}^{an}$ is given as the quotient stack
$[\HH_g/\Gamma_D(N)]$ and in the case $d_g=1$ of principal polarisations we have $\a^{an}_{g, [N]}=[\HH_g/\Gamma_{2g}(N)]$.
\end{proof}

From this we can derive now the \'etale homotopy types of the moduli stacks $\a_{D, [N]}$ and $\a_{g, [N]}$.

\begin{theorem}
\label{ethoN}
Let $g, N\geq 1$ be integers and $D=(d_1, d_2, \ldots d_g)$ be a tuple of positive integers with $d_1|d_2|\ldots |d_g$ and  $\gcd(d_g, N)=1$.
There is a weak homotopy equivalence of pro-simplicial sets
$$\{\a_{D, [N]} \otimes\bar{\Q} \}^{\wedge}_{et} \simeq K(\Gamma_D (N), 1)^{\wedge}.$$
where $^{\wedge}$ denotes Artin-Mazur profinite completion.
In particular for principal polarisations with level $N$ structures we have that
$$\{\a_{g, [N]} \otimes\bar{\Q} \}^{\wedge}_{et} \simeq K(\Gamma_ {2g}(N), 1)^{\wedge}.$$
\end{theorem}

\begin{proof}
The proof is the obvious variation of the proof of Theorem \ref{etho} using Theorem \ref{anahoN} and the comparison theorem for
\'etale homotopy types of Deligne-Mumford stacks.
\end{proof}

As a final consequence we state again the natural interpretation in terms of \'etale fundamental groups of algebraic stacks.

\begin{corollary}
Let $x$ be any point in the stack ${\a}_{g, [N]}\otimes\bar{\Q}$, then we have for the \'etale fundamental group:
$$\pi_1^{et}({\a}_{D, [N]}\otimes\bar{\Q}, x)\cong \Gamma_D(N)^{\wedge}$$
Especially for principally polarised schemes we have that
$$\pi_1^{et}({\a}_{g, [N]}\otimes\bar{\Q}, x)  =  \Gamma_{2g} (N)^{\wedge}.$$
\end{corollary}

Let us also remark that we have the following short exact sequence of \'etale fundamental groups.

\begin{corollary}
Let $x$ be any point in the
stack ${\a}_{g, [N]}\otimes\bar{\Q}$, then there is a short exact sequence of profinite groups
$$1\rightarrow \Gamma_{2g}(N)^{\wedge}  \rightarrow \pi^{et}_{1}({\a}_{g , [N]} \otimes \Q, x) \rightarrow \Gal(\bar{\Q}/\Q)\rightarrow 1.$$
\end{corollary}

Again it turns out that the groups $\Gamma_{2g}(N)$ are not good groups either and so the profinite completion $K(\Gamma_{2g}(N), 1)^{\wedge}$ is {\em not} weakly equivalent to $K(\widehat{\Gamma_{2g} (N)}, 1)$ (cf. \cite{S}, \cite{GJZZ}).

\section{The Torelli Morphism and \'Etale homotopy Types}

We can also compare the \'etale homotopy type of the moduli stack $\a_g$ with that of the moduli stack $\m_g$ of algebraic curves of genus $g$, which was determined by Oda \cite{O}. The relation between the two algebraic stacks is given by the Torelli morphism (cf. \cite[Sec. 7.4]{MFK}, \cite[1.2, 1.3]{MO}, \cite{Vi}). Let $g$ be an integer with $g\geq 2$. Given a family of algebraic curves $\pi:C\rightarrow U$ of genus $g$ we can associate to it its Jacobian given as the relative Picard scheme
$J(C/U)=\Pic^0(C/U)$, which is an abelian scheme over $U$ of relative dimension $g$. There exists a principal polarisation 
$$\vartheta: J(C/U)\rightarrow \hat{J}(C/U)=\Pic^0(J(C)/U)$$
induced by the $\Theta$-divisor. This induces a functor
$$j(U): \m_g(U)\rightarrow \a_g(U), \,\, C/U\mapsto (J(C/U), \vartheta)$$
between the category of sections for every object $U$ in the category of $S$-schemes $(Sch/S)$. Therefore we get a morphism between algebraic stacks, the {\em Torelli morphism}
$$j: \m_g\rightarrow \a_g.$$ 
Torelli's theorem says that for every algebraically closed field $k$ the morphism  $j$ is injective on $k$-valued points (cf. \cite[1.2, 1.3]{MO}). 

Oda \cite{O} showed that the \'etale homotopy type of $\m_g$ is given as the profinite completion of the classifying space of the mapping class group $\Map_g$ of compact Riemann surfaces of genus $g$, i.e. we have a weak homotopy equivalence of pro-simplicial sets 
$$\{\m_g \otimes\bar{\Q} \}^{\wedge}_{et} \simeq K(\Map_g, 1)^{\wedge}.$$

Therefore we get a morphism between pro-simplicial sets induced by the Torelli morphism $j$
$$\{\m_g \otimes\bar{\Q} \}^{\wedge}_{et}\rightarrow \{\a_g \otimes\bar{\Q} \}^{\wedge}_{et}$$
and for the \'etale fundamental groups a morphism between short exact sequences of profinite groups
$$\diagram
1\ar[r] &\Map_g^{\wedge}\ar[r]\ar[d] &\pi_1^{et}(\m_g\otimes \Q, x)\ar[r]\ar[d] &\Gal(\bar{\Q}/\Q) \ar[r]\ar@{=}[d] &1\\
1\ar[r] & \SP(2g, \Z)^{\wedge}\ar[r] &\pi_1^{et}(\a_g\otimes \Q, j(x))\ar[r] &\Gal(\bar{\Q}/\Q) \ar[r] & 1
\enddiagram
$$
where the existence of the first short exact sequence is a direct consequence of the result of Oda.

We can also consider the moduli stack $\m_{g,[N]}$ of smooth algebraic curves of genus $g$ endowed with a level $N$ structure, $N \geq 3$,  where a level $N$ structure on an algebraic curve $C/S$ is a level $N$ structure on the Jacobian variety $J(C/S)$ (cf. \cite{MFK}, \cite{BoP}).

We also have a Torelli morphism 
$$j_N:  \m_{g,[N]} \rightarrow \a_{g, [N]}$$
If $g \geq 3$, the Torelli morphism $J_N$  is $2:1$ on its image outside the hyperelliptic locus and it is ramified on the hyperelliptic locus (cf. \cite{OS}). 
In fact on the Jacobian of a curve one has the automorphism given by multiplication by $-1$ which does not lift to an automorphism of the curve, unless the curve $C$ is hyperelliptic. Indeed in this case one has the hyperelliptic involution on $C$ which acts as -1 on the homology group $H_1(C, \Z/N\Z)$. 

Recall that for $N \geq 3$, the moduli stacks $ \m_{g,[N]}$ are in fact smooth quasi-projective schemes over $\Spec(\Z[N^{-1}])$ (cf. \cite{S1}).

Assume $n \geq 3$, its complex analytification  $ \m_{g,[N]}^{an}$ has the structure of a complex analytic orbifold and indeed of a smooth complex manifold. In fact, recall that the principal polarisation on the Jacobian of a compact Riemann surface $C_g$ of genus $g$ is given by the cup product on the cohomology $H^1(C_g, \Z)$. So if we denote as above by $\Map_g$ the mapping class group, we  have a natural morphism $t:  \Map_g \rightarrow \SP(2g, \Z)$, and also a morphism  $t_N: \Map_g \rightarrow \SP(2g, \Z/N\Z)$. Denote by $\RR_g$ and $\RR_g(N)$ the kernels of the morphisms $t$ and $t_N$ respectively. They both act properly discontinuously on the Teichm\"uller space $\T_g$ and the quotient ${\T}or_g= \T_g/\RR_g$ is called the {\it Torelli space}, while the quotient $\T_g/\RR_g(N)$  is the moduli space of compact Riemann surfaces of genus $g$ with a level $N$ structure and we have a complex analytic orbifold structure on $\m_{g,[N]}^{an}= [\T_g/\RR_g(N)]$ (cf. \cite[Ch. 16]{ACG}, \cite{BoP}). Moreover we have the following diagram for the Torelli morphisms:
$$
\begin{array}{ccc}
\T_g & & \\
\downarrow & & \\
 {\T}or_g & \stackrel{j_{{\T}or}}\rightarrow& \HH_g \\
 \downarrow & & \downarrow\\
\m_{g,[N]}^{an}& \stackrel{j_{N}}\rightarrow& \a_{g, [N]}^{an}  \\
  \downarrow & & \downarrow \\
 \m_g^{an}
 & \stackrel{j}\rightarrow& \a_g^{an} \\
\end{array}
$$
where $j$ is injective, while $j_{{\T}or}$ and $j_N$ are $2:1$ on their respective images and ramified on the hyperelliptic locus (cf. \cite{OS} and also \cite[Ch. 3]{Col} for a nice survey on the general Torelli morphisms). 

It follows in the same way as Oda's result for $\m_g$ that the \'etale homotopy type of the moduli stack $\m_{g, [N]}$ is given as the profinite completion of the classifying space of the discrete group $\RR_{g}(N)$, i.e. we have a weak homotopy equivalence of pro-simplicial sets 
$$\{\m_{g, [N]} \otimes\bar{\Q} \}^{\wedge}_{et} \simeq K(\RR_{g}(N), 1)^{\wedge}.$$
This follows readily from the fact that $\m_{g,[N]}^{an}= [\T_g/\RR_g(N)]$ and that the Teichm\"uller space $\T_g$ is again contractible by using the comparison theorem (Theorem \ref{homdesc}) for homotopy types.

And so we get again a morphism between pro-simplicial sets induced by the Torelli morphism $j_N$, which preserves the level $N$ structures
$$\{\m_{g, [N]} \otimes\bar{\Q} \}^{\wedge}_{et}\rightarrow \{\a_{g, [N]} \otimes\bar{\Q} \}^{\wedge}_{et}$$
and for the \'etale fundamental groups a morphism between short exact sequences of profinite groups
$$\diagram
1\ar[r] &\RR_{g}(N)^{\wedge}\ar[r]\ar[d] &\pi_1^{et}(\m_{g, [N]}\otimes \Q, x)\ar[r]\ar[d] &\Gal(\bar{\Q}/\Q) \ar[r]\ar@{=}[d] &1\\
1\ar[r] & \Gamma_{2g}(N)^{\wedge}\ar[r] &\pi_1^{et}(\a_{g, [N]}\otimes \Q, j_N(x))\ar[r] &\Gal(\bar{\Q}/\Q) \ar[r] & 1
\enddiagram
$$
using the calculations of \'etale homotopy types of the moduli stacks and the short exact sequences for the respective \'etale fundamental groups. 

These observations now allow to study algebro-geometric properties of the Torelli morphisms and questions related to the Schottky problem of characterising the locus of Jacobians among principally polarised abelian varieties in terms of \'etale homotopy types of the associated morphism between algebraic stacks. This will be a theme in a follow-up article.\\

{\it Acknowledgements:} The first author was partially supported by PRIN 2012 MIUR ``Moduli, strutture geometriche e loro applicazioni'', by FIRB 2012 ``Moduli spaces and applications'' and by INdAM (GNSAGA). The second author was partially supported  by INdAM (GNSAGA) and would like to thank the University of Pavia for the kind invitation and hospitality. He also likes to thank the Tata Institute of Fundamental Research (TIFR) in Mumbai, the Centro de Investigaci\'on en Matem\'aticas (CIMAT) in Guanajuato and the University of Leicester for additional support. Finally both authors are very grateful to the referee for suggesting several improvements.

\end{document}